\documentclass[a4paper,12pt]{amsart}

\usepackage{a4wide}
\usepackage{tikz}
\usepackage{url}

\newif\ifdetails
\detailstrue
\newcommand{\DETAIL}[1]%
{\ifdetails\par\fbox{\begin{minipage}{0.9\linewidth}\textit{Detail:}
			#1\end{minipage}}\par\fi}
\newcommand{\TODO}[1]%
{\ifdetails\par\fbox{\begin{minipage}{0.9\linewidth}\textbf{TODO:}
			#1\end{minipage}}\par\fi}

\usepackage{makecell}
\usepackage{relsize}

\newtheorem{lem}{Lemma}
\newtheorem{pro}[subsection]{Proposition}
\newtheorem{thm}[subsection]{Theorem}

\theoremstyle{remark}

\usepackage{caption}
\usepackage{subcaption}
\usepackage{float} % provides H as float placement specifier
\usepackage[pdftex,a4paper,
citecolor = blue, colorlinks=true,urlcolor=blue]{hyperref}
\newcommand{\old}[1]{{}}
\DeclareMathOperator{\N}{\eta}
\usetikzlibrary{arrows}

\title{On the maximum number of connected induced subgraphs of a graph}

\author{Audace A. V. Dossou-Olory}

\address{Audace A. V. Dossou-Olory \\ D{\'e}partement d'Hydrologie et Gestion des Ressources en Eau \\ Institut National de l'Eau \\ and  Centre d'Excellence d'Afrique pour l'Eau et l'Assainissement \\ Universit\'e d'Abomey-Calavi \\ B\'enin}
\email{audace@aims.ac.za}



\begin{document}

\begin{abstract}
We characterise the structure of those graphs of a given order which maximise the number of connected induced subgraphs for seven different graph classes, each with other prescribed parameters like minimum degree, independence number, vertex cover number, vertex connectivity, edge connectivity, chromatic number, number of bridges, thereby contributing to filling a gap in the literature.

\end{abstract}

\maketitle

\section{Introduction and preliminaries}

Let $G$ be a simple graph (undirected, no loops or multiple edges) with finite vertex set $V(G)$ and edge set $E(H)$. The order of $G$ is the cardinality $|V(G)|$. The degree of $u \in V(G)$ is the number of vertices adjacent to $u$ in $G$; we denote by $\delta (G)$ the minimum degree of the vertices of $G$. A subgraph of $G$ is a graph $H$ such that $V(H)\subseteq V(G)$ and $E(H)\subseteq E(G)$. We write $G-e$ (resp. $G-u$) or $G-S$ for the subgraph of $G$ obtained by deleting an edge $e$ (resp. a vertex $u$) or a set of edges/vertices $S$. An induced subgraph of $G$ is a subgraph obtained by deleting a set of vertices. We call $G-(V(G) \backslash S)$ the subgraph induced by $S$, which consists of $S$ and all edges whose endvertices are contained in $S$. An independent set in $G$ is a set of pairwise nonadjacent vertices. Thus, a set $S$ of vertices is an independent set if and only if the subgraph induced by $S$ has no edges. The maximum size of an independent set in $G$ is called its independence number, denoted by $\alpha(G)$.

\medskip
The graph $G$ is connected if every pair of vertices in $G$ belongs to a path; otherwise, $G$ is disconnected. A bridge (or cut-edge) of $G$ is an edge whose deletion increases the number of components (remaining connected parts) of $G$. It is known (see as eg.~\cite{West}) that an edge of $G$ is a bridge if and only if it does not belong to a cycle of $G$. The vertex connectivity of $G$, which we denote by $c(G)$, is the minimum size of a vertex set $S$ such that $G-S$ has more components than $G$. The edge connectivity of $G$, which we denote by $e(G)$, is the minimum size of an edge set $S$ such that $G-S$ has more components than $G$. A vertex cover of $G$ is a set $S\subset V(G)$ that contains at least one endvertex of every edge of $G$; we say that the vertices in $S$ cover $E(G)$. The vertex cover number, denoted by $\beta (G)$, is the minimum size of a vertex cover of $G$.

\medskip
The graph $G$ is $l$-colourable if we can assign one of $l$ colours to each vertex so that adjacent vertices have different colours. If G is $l$-colourable, but not $(l-1)$-colourable, we say that the chromatic number of $G$ is $l$. In other words, the chromatic number of $G$ is the minimum number of colours needed to color the vertices of $G$ such that adjacent vertices have different colours.

A graph $G$ is complete if its vertices are pairwise adjacent; the complete graph of order $n$ is denoted by $K_n$. For all the notation defined on a graph $G$, the context indicates the usage, whether it is followed by the name $G$ or not.

\medskip
An extremal problem asks for the minimum or maximum value of a fonction over a class of objects. In graph theory, we use ``extremal problem" for finding an optimum over a class of graphs. In our case, we are concerned with determining the maximum number of connected induced subgraphs, denoted by $\N(G)$, for simple graphs $G$ with given order and other structural parameters.

Several upper and lower bounds on the number of connected subgraphs or connected induced subgraphs in terms of other graph parameters have been investigated, see for example~\cite{Alokshiya2019, Audacegenral2018, AudaceGirth2018, AudaceCut2019, Audacegenral2023, EricAudace, Sommer, Maxwell2014,  Pandey}. Studies on extremal problem in this direction seems to have begun with Pandey and Patra~\cite{Pandey} on the number of connected (not necessarily induced) subgraphs of both graphs and unicyclic graphs, as a natural extension of the number of subtrees of trees. Among other things, we know by~\cite{Audacegenral2018} that the path uniquely realises the minimum number of connected induced subgraphs among all connected graphs of a given order, and that the maximum is only attained by the complete graph. Moreover, the results in~\cite{Audacegenral2018} can be generalised to graphs with given order and number of components. On the other hand, paper~\cite{AudaceCut2019} deals with the class of all connected graphs with given order and number of cut vertices, as well as the class of all connected graphs with given order and number of pendant vertices. In another paper~\cite{EricAudace}, we studied inequalities that relate the sum of a graph invariant to the same invariant of its complement, also referred to as Nordhaus-Gaddum type inequalities, for the number of connected induced subgraphs of a graph with given order.
 
\medskip
In this note, we characterise the unique extremal graph of a given order which maximise the number of connected induced subgraphs in certain graph classes that have not been considered so far, namely for each of the following parameters: minimum degree, independence number, vertex cover number, vertex connectivity, edge connectivity, chromatic number, number of bridges.

\section{Main results}

%%%%%%%%
For a graph $G$ and $u\in V(G)$, we denote by $\N(G)_{u}$ the number of those connected induced subgraphs of $G$ that contain vertex $u$.
 
Let $n,\delta$ be positive integers such that $n-2 \geq \delta$. We construct the graph $G_{n, \delta}$ by taking one copy of $K_{n-1}$ and adding another vertex to exactly $\delta$ vertices of $K_{n-1}$. Note that $G_{n, \delta}$ has order $n$ and minimum degree $\delta$.

\begin{pro} \label{Prop:mindeg}
For a graph $G \neq G_{n,\delta}$ with order $n$ and minimum degree $\delta$, we have
$$\N(G) < \N(G_{n,\delta})\,.$$
\end{pro}

\begin{proof}
Let $G$ be a graph with order $n$ and minimum degree $\delta$. Fix a vertex $u$ of degree $\delta$ in $G$. Then for $G$ to have the maximum $\N(.)$, the subgraph $G-u$ must be complete. In particular, $G_{n,\delta}$ uniquely realises the maximum $\N(.)$ over those graphs with order $n$ and minimum degree $\delta$.
\end{proof}

%%%%%%

The graph $H_{n,\alpha}$ is obtained by taking disjoint copies of $K_{n-\alpha}$ and $\overline{K_{\alpha}}$ ($\alpha$ independent vertices), then adding an edge between every vertex of $K_{n-\alpha}$ and every vertex of $\overline{ K_{\alpha}}$.

\begin{pro}
For a graph $G \neq H_{n,\alpha}$ with order $n$ and independence number $\alpha$, we have
$$\N(G) < \N(H_{n,\alpha})\,.$$

For a graph $G \neq H_{n,n-\beta}$ with order $n$ and vertex cover number $\beta$, we have
$$\N(G) < \N(H_{n,n-\beta})\,.$$
\end{pro}

%%%%%%%%%%%

\begin{proof}
Let $G$ be a graph with order $n$ and independence number $\alpha$. Fix a set $S$ of $\alpha$ independent vertices in $G$. Then for $G$ to have the maximum $\N(.)$, the subgraph $G-S$ must be complete and every vertex in $S$ must be adjacent to all vertices of $G-S$. In particular, $H_{n,\alpha}$ is the unique graph that realises the maximum $\N(.)$ over all graphs with order $n$ and independence number $\alpha$.

For the second statement of the proposition, it is known (see as eg.~\cite{West}) that $\alpha+\beta=n$, thus completing the proof of the proposition.
\end{proof}

%%%%%%%%

Our next theorem concerns vertex connectivity and edge connectivity. 

\begin{thm} \label{Thm:connectivity}
For a graph $G \neq G_{n, c}$ with order $n$ and vertex connectivity $c$, we have
$$\N(G) < \N(G_{n,c})\,.$$

For a graph $G \neq G_{n, e}$ with order $n$ and edge connectivity $e<n-1$, we have
$$\N(G) < \N(G_{n,e})\,.$$
\end{thm}

\begin{proof}
Let $G$ be a graph with order $n$ and vertex connectivity $c$. Fix a set $S$ of $c$ vertices in $G$ such that $G-S$ is disconnected. Then for $G$ to have the maximum $\N(.)$, the subgraph $G-S$ must have precisely two components (connected parts), say $G_1, G_2$, each of which is a complete graph, and every vertex in $S$ must be adjacent to all vertices of $G-S=G_1 \cup G_2$. So it remains to determine the orders $n_1,n_2$ of $G_1=K_{n_1}$ and $G_2= K_{n_2}$, respectively. We can assume that $n_1 \leq n_2$.

Suppose that $n_1>1$. Fix $u_1 \in V(G_1)$ and construct from $G$ a new graph $G'$ by deleting $u_1$ and adding a new vertex $u_2$ adjacent to all vertices in $S\cup V(G_2)$. Then $G-u_1$ and $G'-u_2$ are isomorphic graphs. By construction of $G$ and $G'$, every subset of $V(G)$ that contains $u_1$ and a vertex in $S$, as well as every subset of $V(G')$ that contains $u_2$ and a vertex in $S$ always induces a connected graph. The number of such subgraphs of $G$ and $G'$ is $(2^c -1)2^{n_1-1} \cdot 2^{n_2}$ and $(2^c -1)2^{n_2} \cdot 2^{n_1-1}$, respectively. It follows that
$$\N(G)_{u_1}=\N(K_{n_1})_{u_1}+(2^c -1)2^{n_1-1} \cdot 2^{n_2}= 2^{n_1-1} + (2^c -1)2^{n_1-1+n_2}  $$
and
$$\N(G')_{u_2}=\N(K_{n_2+1})_{u_2}+(2^c -1)2^{n_2} \cdot 2^{n_1-1}=  2^{n_2} +(2^c -1)2^{n_2+n_1-1} \,.$$
Therefore, we obtain
\begin{align*}
\N(G')- \N(G)= \N(G')_{u_2} - \N(G)_{u_1}= 2^{n_2} -2^{n_1-1} >0\,,
\end{align*}
which shows that $\N(G') > \N(G)$. Hence $n_1=1$ for $G$ to have the maximum $\N(.)$. On the other hand, $S$ must induce a complete graph: this implies that the graph realising the maximum $\N(.)$ is indeed $G_{n, c}$.

\medskip
For a proof of the second statement of the theorem, first note that if $u$ is the unique vertex of degree $x < n-2$ in $G_{n, x}$, then $G_{n, x}-u=K_{n-1}$ whereas
$$\N(G_{n, x})_u= 1+(2^x-1)2^{n-1-x} = 1+ 2^{n-1} - 2^{n-1-x}$$ counts precisely the number of $u$-containing connected induced subgraphs of $G_{n, x}$. Thus $\N(G_{n, x})=\N(K_{n-1}) + \N(G_{n, x})_u$ is a strictly increasing function in $x$.

Now let $G$ be a graph with order $n$ and edge connectivity $e$. Denote by $c$ the vertex connectivity of $G$. Then we have $\N(G)\leq \N(G_{n, c}) \leq \N(G_{n, e}) $, where the first inequality stems from the first statement of the theorem, and the second inequality holds by Whitney's theorem~\cite{Whitney} that $c\leq e$. This completes the proof of the theorem. 
\end{proof}

%%%%%%%%%

The Tur\`an graph $T_{n,l}$ is a complete $l$-partite graph of order $n$ in which any two partition sets differ in cardinality by at most one. This famous graph appears in many extremal graph theory problems.

\begin{thm} 
For a graph $G \neq T_{n,l}$ with order $n$ and chromatic number $l$, we have
$$\N(G) < \N(T_{n,l})\,.$$
\end{thm}

\begin{proof}
Let $G$ be a graph with order $n$ and chromatic number $l$. Partition the vertices of $G$ according to their colours. Then each of the $l$ partition sets (colour classes) is an independent set whose sizes are $n_1,n_2,\ldots,n_l$ with $n_1 + n_2 + \cdots + n_l=n$. Thus, for $G$ to have the maximum $\N(.)$, it has to be a complete $l$-partite graph with these partite sets. In particular, only $K_n=T_{n,n}$ realises the maximum if $l=n$. In what follows, we assume $l<n$. Let $A$ and $B$ be two partite sets having the greatest and smallest cardinality, respectively, among all the $l$ partite sets of $G$. Then we have $|A|> 1$. Let us show that $|A|-1 \leq |B| \leq |A|$. 

Suppose not, fix $a \in A$. Contruct from $G$ a new graph $G'$ by deleting vertex $a$, then taking a new vertex $b$ with the same colour as those in $B$, and adding an edge between $b$ and every vertex in $V(G) \backslash (B \cup \{a,b\})$. Note that $\N(G')- \N(G)=\N(G')_b -\N(G)_a $ since the graphs $G-a$ and $G'-b$ are isomorphic by construction. Moreover, every subset of vertices of $G$ that contains both $a$ and an element of $V(G) \backslash A$, as well as every subset of vertices of $G'$ that contains both $b$ and an element of $V(G') \backslash (B \cup \{b\})$ always induces a connected graph. On the other hand, $A$ and $B \cup \{b\}$ are independent sets in $G$ and $G'$, respectively. Theorefore, we get $\N(G)_a=1+ (2^{|V(G)| -|A|}-1) 2^{|A|-1}$ and $\N(G')_b=1+ (2^{|V(G')| -|B|-1}-1) 2^{|B|}$. Hence
\begin{align*}
\N(G')- \N(G)=\N(G')_b -\N(G)_a &=2^{|V(G')|-1} - 2^{|B|}+  2^{|A|-1} - 2^{|V(G)|-1} \\
&= 2^{|A|-1} - 2^{|B|} >0
\end{align*}
by the inequality $|B|< |A|-1$. This is a contradiction to the choice of $G$ that $\N(G') \leq \N(G)$. Consequently, for $G$ to have the maximum $\N(.)$, it has to be a complete $l$-partite graph with $|A|-1 \leq |B| \leq |A|$. By the choice of the sets $A$ and $B$, we conclude that the graph realising the maximum $\N(.)$ is indeed the Tur\`an graph $T_{n,l}$.
\end{proof}

%%%%%%%%
For our next theorem, we need to start with a lemma which can also be found in~\cite{AudaceGirth2018,Audacegenral2023}. 

\begin{lem}[Lemma 2.1~\cite{AudaceGirth2018}, Lemma 1~\cite{Audacegenral2023}]\label{singleBranch}

Let $L,M,R$ be three non-trivial connected graphs whose vertex sets are pairwise disjoints. Let $l \in V(L),~ r\in V(R)$ and $u,v \in V(M)$ be fixed vertices such that $u \neq v$. Denote by $G$ the graph obtained from $L,M,R$ by identifying $l$ with $u$, and $r$ with $v$. Similarly, let $G'$ (resp. $G''$) be the graph obtained from $L,M,R$ by identifying both $l,r$ with $u$ (resp. both $l,r$ with $v$); see Figure~\ref{diagGGpGpp} for a diagram of these graphs. Then it holds that
	\begin{align*}
	\N(G') > \N(G) ~~ \text{or} ~~ \N(G'') > \N(G)\,.
	\end{align*}
	\begin{figure}[htbp]\centering
		\definecolor{qqqqff}{rgb}{0.,0.,1.}
		\resizebox{0.7\textwidth}{!}{%
		\begin{tikzpicture}[line cap=round,line join=round,>=triangle 45,x=1.0cm,y=1.0cm]
		rectangle (17.25469285535401,15.543118314192673);
		\draw [dash pattern=on 2pt off 2pt] (8.02,12.78)-- (10.02,12.78);
		\draw [dash pattern=on 2pt off 2pt] (10.02,12.78)-- (10.,12.);
		\draw [dash pattern=on 2pt off 2pt] (10.,12.)-- (10.02,11.26);
		\draw [dash pattern=on 2pt off 2pt] (10.02,11.26)-- (8.,11.26);
		\draw [dash pattern=on 2pt off 2pt] (8.,11.26)-- (8.,12.);
		\draw [dash pattern=on 2pt off 2pt] (8.,12.)-- (8.02,12.78);
		\draw [dash pattern=on 2pt off 2pt] (6.48,12.56)-- (8.,12.);
		\draw [dash pattern=on 2pt off 2pt] (6.48,11.48)-- (8.,12.);
		\draw [dash pattern=on 2pt off 2pt] (6.48,12.56)-- (6.48,11.48);
		\draw [dash pattern=on 2pt off 2pt] (10.,12.)-- (11.48,12.48);
		\draw [dash pattern=on 4pt off 4pt] (10.,12.)-- (11.46,11.46);
		\draw [dash pattern=on 2pt off 2pt] (11.48,12.48)-- (11.46,11.46);
		\draw (8.002875389073733,12.308886980787792) node[anchor=north west] {$l,u$};
		\draw (9.126357554291436,12.19457622267538) node[anchor=north west] {$v,r$};
		\draw (6.736460949518799,12.287683803799489) node[anchor=north west] {$L$};
		\draw (10.70036563986977,12.304751529462257) node[anchor=north west] {$R$};
		\draw [dash pattern=on 2pt off 2pt] (4.,10.)-- (6.,10.);
		\draw [dash pattern=on 2pt off 2pt] (6.,10.)-- (5.98,9.22);
		\draw [dash pattern=on 2pt off 2pt] (5.98,9.22)-- (6.,8.48);
		\draw [dash pattern=on 2pt off 2pt] (6.,8.48)-- (3.98,8.48);
		\draw [dash pattern=on 2pt off 2pt] (3.98,8.48)-- (3.98,9.22);
		\draw [dash pattern=on 2pt off 2pt] (3.98,9.22)-- (4.,10.);
		\draw [dash pattern=on 2pt off 2pt] (3.12,10.28)-- (3.98,9.22);
		\draw [dash pattern=on 2pt off 2pt] (2.6,9.6)-- (3.98,9.22);
		\draw [dash pattern=on 2pt off 2pt] (3.12,10.28)-- (2.6,9.6);
		\draw (4.031552359486242,9.502599906980058) node[anchor=north west] {$l,u,r$};
		\draw (6.006612285785838,9.41691066509247) node[anchor=north west] {$v$};
		\draw (3.008420807842294,10.0050302314765) node[anchor=north west] {$L$};
		\draw (3.087042324067116,8.845683517584332) node[anchor=north west] {$R$};
		\draw [dash pattern=on 2pt off 2pt] (12.,10.)-- (14.,10.);
		\draw [dash pattern=on 2pt off 2pt] (14.,10.)-- (13.98,9.22);
		\draw [dash pattern=on 2pt off 2pt] (13.98,9.22)-- (14.,8.48);
		\draw [dash pattern=on 2pt off 2pt] (14.,8.48)-- (11.98,8.48);
		\draw [dash pattern=on 2pt off 2pt] (11.98,8.48)-- (11.98,9.22);
		\draw [dash pattern=on 2pt off 2pt] (11.98,9.22)-- (12.,10.);
		\draw [dash pattern=on 2pt off 2pt] (13.98,9.22)-- (14.9,10.32);
		\draw [dash pattern=on 4pt off 4pt] (13.98,9.22)-- (15.42,9.66);
		\draw [dash pattern=on 2pt off 2pt] (14.9,10.32)-- (15.42,9.66);
		\draw (12.000266144323993,9.432599906980058) node[anchor=north west] {$u$};
		\draw (12.761465421630708,9.532775213766935) node[anchor=north west] {$v,l,r$};
		\draw (14.542891846445562,10.147962505813732) node[anchor=north west] {$R$};
		\draw [dash pattern=on 2pt off 2pt] (2.66,8.42)-- (3.98,9.22);
		\draw [dash pattern=on 2pt off 2pt] (3.44,7.86)-- (3.98,9.22);
		\draw [dash pattern=on 2pt off 2pt] (2.66,8.42)-- (3.44,7.86);
		\draw (8.654102536581874,12.838560337733874) node[anchor=north west] {$M$};
		\draw (4.70415799076956,10.0550302314765) node[anchor=north west] {$M$};
		\draw (12.718182533719721,10.0550302314765) node[anchor=north west] {$M$};
		\draw [dash pattern=on 2pt off 2pt] (13.98,9.22)-- (15.44,8.52);
		\draw [dash pattern=on 2pt off 2pt] (14.78,7.92)-- (13.98,9.22);
		\draw [dash pattern=on 2pt off 2pt] (15.44,8.52)-- (14.78,7.92);
		\draw [dash pattern=on 2pt off 2pt] (15.44,8.52)-- (14.78,7.92);
		\draw [dash pattern=on 2pt off 2pt] (15.44,8.52)-- (14.78,7.92);
		\draw (14.521338055883508,8.817237308146387) node[anchor=north west] {$L$};
		\draw (8.682548746019819,11.18282315480661) node[anchor=north west] {$G$};
		\draw (4.618293442095093,8.37342849987477) node[anchor=north west] {$G'$};
		\draw (12.675250259382487,8.330496225537535) node[anchor=north west] {$G''$};
		
		\draw [fill=black] (8.02,12.78) circle (0.5pt);
		\draw [fill=black] (10.02,12.78) circle (0.5pt);
		\draw [fill=black] (8.,11.26) circle (0.5pt);
		\draw [fill=black] (10.02,11.26) circle (0.5pt);
		\draw [fill=black] (8.,12.) circle (2.0pt);
		\draw [fill=black] (10.,12.) circle (2.0pt);
		\draw [fill=black] (6.48,12.56) circle (0.5pt);
		\draw [fill=black] (6.48,11.48) circle (0.5pt);
		\draw [fill=qqqqff] (11.48,12.48) circle (0.5pt);
		\draw [fill=qqqqff] (11.46,11.46) circle (0.5pt);
		\draw [fill=black] (4.,10.) circle (0.5pt);
		\draw [fill=black] (6.,10.) circle (0.5pt);
		\draw [fill=black] (3.98,8.48) circle (0.5pt);
		\draw [fill=black] (6.,8.48) circle (0.5pt);
		\draw [fill=black] (3.98,9.22) circle (2.0pt);
		\draw [fill=black] (5.98,9.22) circle (2.0pt);
		\draw [fill=black] (3.12,10.28) circle (0.5pt);
		\draw [fill=black] (2.6,9.6) circle (0.5pt);
		\draw [fill=black] (12.,10.) circle (0.5pt);
		\draw [fill=black] (14.,10.) circle (0.5pt);
		\draw [fill=black] (11.98,8.48) circle (0.5pt);
		\draw [fill=black] (14.,8.48) circle (0.5pt);
		\draw [fill=black] (11.98,9.22) circle (2.0pt);
		\draw [fill=black] (13.98,9.22) circle (2.0pt);
		\draw [fill=black] (15.44,8.52) circle (0.5pt);
		\draw [fill=black] (14.78,7.92) circle (0.5pt);
		\draw [fill=qqqqff] (14.9,10.32) circle (0.5pt);
		\draw [fill=qqqqff] (15.42,9.66) circle (0.5pt);
		\draw [fill=black] (3.44,7.86) circle (0.5pt);
		\draw [fill=black] (2.66,8.42) circle (0.5pt);
		\end{tikzpicture}}
		\caption{The graphs $G,G',G''$ described in Lemma~\ref{singleBranch}.}\label{diagGGpGpp}
	\end{figure}
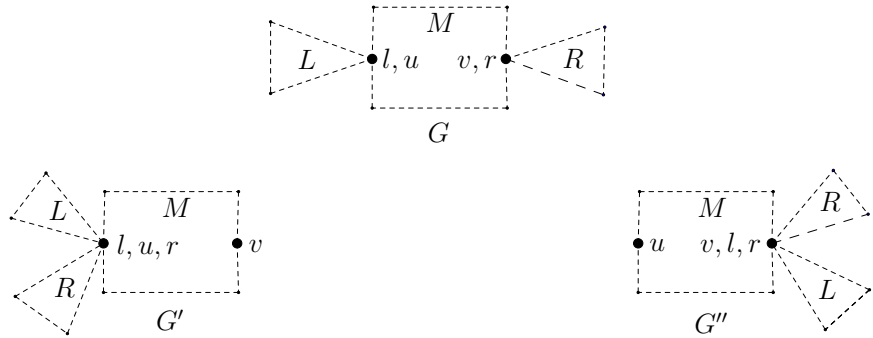

\end{lem}

\medskip
An edge of a graph $G$ with an endvertex of degree $1$ is called a pendant edge of $G$. The star of order $n$ is denoted by $S_{n}$.

For positive integers $b, n$ such that $b< n-2$, define $J_{n,b}$ to be the graph obtained by identifying one vertex of $K_{n-b}$ with the central vertex of $S_{b+1}$. For $b=n-2$, the graph $J_{n,b}$ is defined as that obtained by inserting one vertex into an edge of $S_{n-1}$.
%%%%%

\begin{thm}
For a graph $G \neq J_{n, b}$ with order $n$ and $b$ bridges, we have
$$\N(G) < \N(J_{n,b})\,.$$
\end{thm}

\begin{proof}
Let $G$ be a graph with order $n>2$ and $b$ bridges. Fix a brigde with endvertices $u,v$ in $G$. Then the specialisation $M=uv$ in Lemma~\ref{singleBranch} shows that all edges of $G$ preserve their status (bridges/non bridges) in both $G'$ and $G''$.  Moreover, the bridge $uv$ becomes a pendant edge in both $G'$ and $G''$. Thus it suffices to prove the theorem for $G$ in the class of those graphs with order $n$ and $b$ pendant edges (or pendant vertices). 

Fortunately, in~\cite{AudaceCut2019,EricAudace} we determined the graph structure that maximises the number of connected induced subgraphs, given both order and number of pendant vertices: If $b<n-2$, the extremal graph can be obtained by identifying one vertex of $K_{n-b}$ with the central vertex of $S_{b+1}$; if $b=n-2$, it can be obtained by inserting one vertex into an edge of $S_{n-1}$. The theorem follows.
\end{proof}

\section{Concluding comments}

Among all graphs with order $n$ and minimum degree $\delta$, the minimum number of connected induced subgraphs is attained by the path if $\delta=1$~\cite{Audacegenral2018} and by the cycle if $\delta=2$~\cite{Audacegenral2018, AudaceCut2019}. This motivates the following problem: What is the minimum number of connected induced subgraphs among all graphs with given order and minimum degree $\delta >2$? Recall that Proposition~\ref{Prop:mindeg} answers the maximisation counterpart of this problem.

Theorem~\ref{Thm:connectivity} establishes that the graph with given order and connectivity, and maximum number of connected induced subgraphs is the same in the case of vertex- and edge-connectivity. As it should be mentionned (see~\cite{Audacegenral2018, AudaceCut2019}), the path (resp. cycle) also has minimum number of connected induced subgraphs among all graphs with order $n$ and connectivity $1$ (resp. connectivity $2$). Other relations between number of connected induced subgraphs and graph connectivity have not been established so far. Therefore, the other natural question at this point is to ask for the general case where connectivity is more than $2$. It appears to us that even the case of connectivity $3$ can be very difficult.

\end{document}